\documentclass{amsart} 

\usepackage{mathrsfs,enumerate,amsmath,amssymb,amsthm,verbatim}
\usepackage{thomas}
\usepackage{url}
\usepackage[latin1]{inputenc}
\usepackage{mdwlist}
\usepackage{pdfsync}

\makeatletter
\@ifundefined{texorpdfstring}{\newcommand{\texorpdfstring}[2]{#2}}{}
\makeatother

\title{Luzin's Condition (N) and Modulus of Continuity}
\author{P.~Koskela}
\address{Matematiikan ja tilastotieteen laitos, PL 35 (MaD), 40014 Jyv\"askyl\"an yliopisto, Finland}
\email{pekka.j.koskela@jyu.fi}
\author{J.~Mal\'y}
\address{Department of Mathematical Analysis,
Faculty of Mathematics and Physics,
Charles University,
Sokolovsk\'a 83, 18675 Praha 8,
Czech Republic}
\email{jan.maly@mff.cuni.cz}
\author{T.~Z\"urcher}
\address{Matematiikan ja tilastotieteen laitos PL 35 (MaD), 40014 Jyv\"askyl\"an yliopisto, Finland}
\email{thomas.t.zurcher@jyu.fi}
\thanks{The first and third authors acknowledge the support by the Academy of Finland grant numbers
131477 and 251650, respectively.
The research of the second author was supported by the grant
GA\,\v{C}R P201/12/0436}

\def\ball{\overline B}
\def\en{\mathbb N}
\def\er{\mathbb R}
\def\rd{\er^d}
\def\rn{\er^n}
\let\norm=\abs
\DeclareMathOperator{\loc}{loc}

\def\ordmod{\omega_f^*}
\def\medmod{\omega_f}

\def\H{\mathcal H}

\def\eqn#1$$#2$${\begin{equation}\label#1#2\end{equation}}

\newcommand{\measure}[1]{\lvert#1\rvert}

\newcommand{\lopen}[2]{(#1,#2]}
\newcommand{\ropen}[2]{[#1,#2)}
\usepackage{reference}

\begin{document}

\begin{abstract}
In this paper, we establish Luzin's condition~(N)
for mappings in certain Sobolev-Orlicz spaces with certain moduli of continuity. Further, given a mapping in these Sobolev-Orlicz spaces, we give bounds on the size of the exceptional set where Luzin's condition~(N) may fail. If a mapping violates Luzin's condition~(N), we show that there is a Cantor set of measure zero that is mapped to a set of positive measure.
\end{abstract}

\maketitle

\section{Introduction}
In this introduction, we assume for simplicity that our mappings are continuous. In the entire paper, we assume that $n\geq 2$.

As seen by
the constructions of J.~Mal\'y and O.~Martio, for each $n\geq 2$ there exists a continuous mapping $f\colon \Omega\to \field{R}^n$ in $W^{1,n}(\Omega;\field{R}^n)$ such that
$f$ maps a set of $n$\nobreakdash-dimensional Lebesgue measure zero onto an $n$\nobreakdash-dimensional cube, see \cite[Section~5]{mal_lusins_1995}.
However, if we require better integrability for the differential, then $f$ satisfies \emph{Luzin's condition~(N)}: the image of
each set of zero $n$\nobreakdash-dimensional Lebesgue measure is also of
measure zero. More precisely, if additionally
\begin{equation}\label{Orlicz integrability}
\int_\Omega \norm{Df}^n\log^{\lambda}(e+\norm{Df})\, dy<\infty
\end{equation}
for some $\lambda>n-1$, then $f$ satisfies Luzin's condition~(N), see \cite[Example~5.3]{kauhanen_functions_1999}. For $\lambda=n-1$, there exist continuous mappings in $W^{1,n}(\field{R}^n,\field{R}^n)$ satisfying \eqref{Orlicz integrability} and violating Luzin's condition~(N), see \cite[Theorem~5.2]{kauhanen_functions_1999}. Such a mapping $f$ cannot be H\"older continuous. Indeed, in
\cite{mal_lusins_1995} the authors showed that H\"older continuous mappings in $W^{1,n}(\Omega;\field{R}^n)$ satisfy Luzin's condition~(N), see their Theorem~C. Based on this result, the authors further showed in Theorem~G that for each (continuous) mapping in $W^{1,n}(\Omega;\field{R}^n)$, there is a zero-dimensional set such that the mapping satisfies Luzin's condition~(N) outside this set.

In this paper, we consider mappings that satisfy \eqref{Orlicz integrability} for $0\leq\lambda\leq n-1$ and have a modulus of continuity that is slightly weaker than H\"older continuity. Our first result reads as follows.

\begin{theorem}\tlabel{t:mainintro}
Let $0\leq \lambda\leq n-1$ and $\alpha=1-\frac{\lambda}{n-1}$. Suppose $f\in W^{1,1}_\textnormal{loc}(\Omega;\field{R}^d)$,
\begin{equation*}
\int_\Omega \norm{Df}^n\log^{\lambda}(e+\norm{Df})\; dy<\infty,
\end{equation*}
and $f$ has the  modulus of continuity
\begin{equation*}
\norm{f(x)-f(y)}\leq
\begin{cases}
\exp(-\mu\log^\alpha(1/\abs{x-y})), & \alpha>0,\\
\log^{-\mu}(1/\abs{x-y}), & \alpha=0,
\end{cases}
\end{equation*}
for some $\mu>0$. Then $f$ satisfies Luzin's condition~(N) in the sense
that sets of zero $n$-dimensional Lebesgue measure get mapped to sets of zero
$n$-dimensional Hausdorff measure.
\end{theorem}

We prove a corresponding result about the size of the exceptional set as well. Our continuity assumption below can actually be removed provided we choose a quasicontinuous representative, see Section~\ref{Sec exceptional set}.
\begin{theorem}\tlabel{Size of exceptional set}
Let $\Omega\subset \field{R}^n$ open, $n\geq 2$. Suppose $f\in W^{1,n}(\Omega;\field{R}^d)$ is continuous, $0\leq\lambda\leq n-1$, and $\alpha=1-\frac{\lambda}{n-1}$. If
\begin{equation*}
\int_\Omega \norm{Df}^n\log(e+\norm{Df})^\lambda\, dy<\infty,
\end{equation*}
then there exists a set $E$ such that $f$ satisfies Luzin's condition~(N) in $\Omega\setminus E$, and $\mathcal{H}^\varphi(E)=0$ whenever $\varphi$ is of the form
\begin{equation*}
\varphi(r)=
\begin{cases}
\exp(-\gamma\log^{\alpha}(\tfrac{1}{r})),&\alpha>0,\\
\log^{-n-\gamma}(\tfrac{1}{r}),&\alpha=0,
\end{cases}
\qquad
0<r<1,
\end{equation*}
for some $\gamma>0$.
\end{theorem}

Regarding necessity of our modulus of continuity, we construct the following example.

\begin{example}\tlabel{ExampleModulusIntro}
Let $0\leq \lambda<n-1$ and $0< \alpha<\frac{n-1-\lambda}{n}$. Then there is a mapping $f\in W^{1,1}(\field{R}^n;\field{R}^d)$
such that
\begin{equation*}
\int_{\field{R}^n} \norm{Df}^n\log^\lambda(e+\norm{Df})\, dy<\infty
\end{equation*}
and $f$ has modulus
of continuity no worse than
\begin{equation*}
\Psi(t)=C\exp\bigl(-\log 2\log^\alpha(\tfrac{1}{t})\bigr)
\end{equation*}
and violates Luzin's condition~(N).
\end{example}
The \nref{ExampleModulusIntro} above shows that the logarithmic scale in \tref{t:mainintro} is essentially sharp. However, we have a slight mismatch between the exponent $\alpha$ in \tref{t:mainintro} and the one in the \nref{ExampleModulusIntro}. As $n$ gets larger, the gap gets smaller. We do not know if the positive result or the example could be improved. For the endpoint
case $\lambda=n-1$ see the discussion at the end of Section~\ref{Section Example}.

In \tref{ExampleModulusIntro} a compact, perfect, totally disconnected set of measure zero is mapped onto a set of positive measure. We further show that this phenomenon is ubiquitous, i.e.\ if we have a continuous mapping that maps a set of measure zero onto  a set of positive measure, then there is a compact, perfect, totally disconnected set that is blown up as well. Since the verification of Luzin's condition~(N) then only needs to focus on these Cantor sets, we would like to know if additional regularity can be required for the blown up sets. If $f\in W^{1,n}(\Omega;\field{R}^d)$,
then we can actually find a zero-dimensional Cantor set that is blown up. We would like to know if one could require the set to be porous as well, and in dimension two if the set could be required to lie on a quasicircle. For the significance of these questions see \cite{KK}, \cite{KZ}.

Some of our results were announced in \cite{KMZ}. The paper is organized as follows. The following section contains notations and basic definitions. Section~\ref{Section Condition N} deals with Luzin's condition~(N). Especially, we prove a slightly stronger statement than \tref{t:mainintro}. We continue by constructing \tref{ExampleModulusIntro} in Section~\ref{Section Example}. In the penultimate section, we study the size of the exceptional set. Our reasoning is closed by regularity considerations in Section~\ref{Section Regularity}.

\section{Notation and basic definitions}\label{Section Notation}
We will use open and closed balls. To distinguish them, we use an overbar for closed balls and the plain symbol for open balls.
\begin{definition}[gauge function]
A \emph{(Hausdorff) gauge function} $h\colon [0,\infty]\to [0,\infty]$ is a
nondecreasing function that is positive for $t>0$ and continuous from the right.
\end{definition}

\begin{example}
The function $\varphi \colon [0,\infty]\to [0,\infty]$ defined as
\begin{equation*}
\varphi(t)=
\begin{cases}
0, & t=0, \\
D \exp(-C\log^\alpha(\frac{1}{t})), & 0<t<1,\\
\infty, & t\geq 1,\\
\end{cases}
\end{equation*}
where $\alpha,\, C,\, D>0$, is a gauge function.
\end{example}

\begin{definition}[Hausdorff measure]
Let $A\subset \field{R}^n$, $h$ be a gauge function, and $0<\delta\leq \infty$. Define
\begin{align*}
\mathcal{H}^h_\delta&:=\inf \Bigl\{\sum_{j=1}^\infty h(\diam C_j):\; A\subset \cup_{j=1}^\infty C_j,\; \diam C_j\leq\delta \Bigr\},\\
\intertext{and the \emph{Hausdorff $\H^h$\nobreakdash-measure} as}\\
\mathcal{H}^h(A)&:=\lim_{\delta\to 0}\mathcal{H}^\delta(A).
\end{align*}
If $h(t)=t^s$ for some $s\geq 0$, we simply write $\mathcal{H}^s$ instead of $\mathcal{H}^{t^s}$.
The \emph{Hausdorff dimension} $\dim_H(A)$ of a set $A\subset \Omega$ is defined as $\inf\{s\geq 0\colon \mathcal{H}^s(A)=0\}$.
\end{definition}
In what follows $\Omega$ is an open set in $\field{R}^n$.

\begin{definition}[(median) modulus of continuity] Let $f:\Omega\to\rd$ be a measurable function and $x\in\Omega$.
We define the \emph{modulus of continuity}
$$
\ordmod(x,r)=\inf\{s\ge 0: \norm{f-f(x)}\le s \text{ on }B(x,r)\}
$$
and the \emph{median modulus of continuity}
$$
\medmod(x,r)=\inf\{s\ge 0: |\{y\in B(x,r):\norm{f(y)-f(x)}\ge s\}| < \tfrac12|B(x,r)|\}.
$$
Obviously $\medmod(x,r)\le \ordmod(x,r)$, so that results using $\medmod$ are better.
\end{definition}

\section{Pointwise modulus of continuity and Luzin's condition~(N)}\label{Section Condition N}
The following result is a  more general version of \tref{t:mainintro}. The current section is devoted to its proof.
\begin{theorem}\tlabel{t:main}
Let $0\leq\lambda\le n-1$ and $\alpha=1-\frac{\lambda}{n-1}$.
Let $f\in W_{\loc}^{1,1}(\Omega;\rd)$ and $N\subset \Omega$ be a Lebesgue null set.
Suppose that for each $x\in N$ there exist $R=R(x)>0$ and $\mu=\mu(x)>0$ such that
$B(x,R)\subset\Omega$ and
$$
\medmod(x,r)\le
\begin{cases}
 \exp\bigl(-\mu\log^{\alpha}\bigl(1/r\bigr)\bigr) ,&\alpha>0,\\
\log^{-\mu}(1/r),&\alpha=0,
\end{cases}
\qquad 0<r\le R.
$$
If
$$
\int_{\Omega}|Df|^n\log^{\lambda}(e+|Df|)\,dy < \infty,
$$
then $\H^n(f(N))=0$.
\end{theorem}
\subsection{Key estimates}
\begin{lemma}\label{l:riesz}
Let $u\in W^{1,1}(B(0,R))$. Let $M$ be the median of $u$ in $B(0,R)$
and $m$ be the median of $u$ in $B(0,r)$, where $0<r\le R$. Then
\eqn{Riesz}
$$
|M-m|\le C(n)r^{1-n}\int_{B(0,r)}|Du|\,dy
+C(n)\int_{B(0,R)\setminus B(0,r)}|y|^{1-n}|Du|\,dy.
$$
\end{lemma}

\begin{proof}
We find $j\in\en$ such that $2^{j-1}\le R/r<2^{j}$ and
consider a chain of balls $B(0,r_k)$,
$k=0,\dots,j$, where $r_0=R$ and
$r_k=r2^{j-k}$ for $k=1,\dots,j$. Let $m_k$ denote the median of $u$
in $B(0,r_k)$. For each $k=0,\dots,j-1$ and $c\in\er$ we have
$$
\aligned
\tfrac{|B(0,r_{k+1})|}2\;|m_k-m_{k+1}|
&\le \tfrac{|B(0,r_{k})|}2\;|m_k-c|+
\tfrac{|B(0,r_{k+1})|}2\;|m_{k+1}-c|
\\&\le \int_{B(0,r_k)}|u-c|\, dy.
\endaligned
$$
Then we apply the Poincar\'e inequality with the choice $c=u_{B(0,r_k)}$
and obtain
$$
\aligned
|m_k-m_{k+1}|&\le Cr_k^{-n}\int_{B(0,r_k)}|u-c|\, dy\le
Cr_k^{1-n}\int_{B(0,r_k)}|Du|\, dy
\\&=
Cr_k^{1-n}\Bigl(\sum_{i=k}^{j-1}\int_{B(0,r_i)\setminus B(0,r_{i+1})}|Du|\, dy
+\int_{B(0,r_j)}|Du|\, dy\Bigr).
\endaligned
$$
Summing over $k$ we obtain
$$
\aligned
|M-m|&\le
C\sum_{k=0}^{j}
\Bigl(\sum_{i=k}^{j-1}r_k^{1-n}\int_{B(0,r_{i})\setminus B(0,r_{i+1})}|Du|\,dy
+r_k^{1-n}\int_{B(0,r_j)}|Du|\, dy\Bigr)
\\&=
C\sum_{i=0}^{j-1}\sum_{k=0}^{i}
r_k^{1-n}\int_{B(0,r_{i})\setminus B(0,r_{i+1})}|Du|\,dy+
C\sum_{k=0}^jr_k^{1-n}\int_{B(0,r_j)}|Du|\, dy.
\endaligned
$$
Since
$$
\sum_{k=0}^ir_k^{1-n}\le C|y|^{1-n},\qquad y\in B(0,r_{i})\setminus B(0,r_{i+1}),\;i=0,\dots,j-1
$$
and
$$
\sum_{k=0}^jr_k^{1-n}\le Cr^{1-n},
$$
the estimate follows.
\end{proof}

\begin{lemma}\label{l:hoelder}
Let $f\in W^{1,n}(B(0,R);\rd)$. Let $M$ be the median of $|f|$ in $B(0,R)$
and $m$ be the median of $|f|$ in $B(0,r)$, where $0<r\le R$.
Suppose $m\le M$. Then
\begin{equation*}\label{Pseudo RR}
(M-m)^n\le C(n)(1+\log(R/r))^{n-1}\int_{B(0,R)\cap \{|f|<M\}}|Df|^n\,dy.
\end{equation*}
\end{lemma}

\begin{proof} We apply Lemma \ref{l:riesz} to $u=\min\{|f|,M\}$ and obtain
$$
M-m\le C\int_{B(0,R)}v^{1-n}|Du|\,dy\le C\int_{B(0,R)\cap \{|f|<M\}}v^{1-n}|Df|\,dy,
$$
where $v=\max\{r,|x|\}$. By the H\"older inequality we conclude that
$$
M-m\le C\Bigl(\int_{B(0,R)}v^{-n}\,dy\Bigr)^\frac{n-1}{n}
\Bigl(\int_{B(0,R)\cap \{|f|<M\}}|Df|^n\,dy\Bigr)^{\frac1n}.
$$
Now, it remains to notice that
$$
\int_{B(0,R)}v^{-n}\,dy\le C(1+\log(R/r)).
$$
\end{proof}

\begin{lemma}\tlabel{l:elementary}
Let $q$ be a positive integer, $0\leq\alpha<1$, and $\tau>0$. Then there exists
a constant $C=C(\tau,\alpha)>0$
with the following property:
Let $\{a_j\}_{j=q}^{\infty}$ be a nondecreasing sequence of positive real numbers
and
$$
b_j=
\begin{cases}
\tau j^{1/\alpha},&\alpha>0,\\
e^{\tau j},& \alpha=0.
\end{cases}
$$
Suppose
\eqn{borders}
$$
j+1\le a_j\le b_j,\qquad j=q,\,q{+}1,\,q{+}2,\dots.
$$
Then there exists an integer $k\ge q$ such that
\eqn{recurrent}
$$
 a_{k+1}-a_k\le C(a_k-k)^{1-\alpha}.
$$
\end{lemma}

\begin{proof}
Suppose that \eqref{recurrent} fails for a fixed $C$. Write $\nu=\tau+1$.

First, let $0<\alpha<1$.
Towards a contradiction with \eqref{borders}, it suffices to verify
\begin{equation}\label{InequalityToProve}
a_k\geq \nu(k-q)^{1/\alpha}+k+1,\quad k=q,\,q+1,\,q+2,\ldots
\end{equation}
Assume \eqref{InequalityToProve} is true for some $k$; it clearly holds for $k=q$. Then
\begin{equation}\label{ind}
a_{k+1}\geq C(a_k-k)^{1-\alpha}+a_k\geq C(\nu(k-q)^\frac{1}{\alpha}+1)^{1-\alpha}+\nu(k-q)^{1/\alpha}+k+1.
\end{equation}
We see that we succeed in proving \eqref{InequalityToProve} if
\begin{equation*} C(\nu(k-q)^\frac{1}{\alpha}+1)^{1-\alpha}+\nu(k-q)^{1/\alpha}+k+1
\geq \nu(k+1-q)^{1/\alpha}+k+2.
\end{equation*}
We are fine if there is $0<C<\infty$ such that
\begin{equation*}
C\geq \sup_{x\geq q} \frac{\nu(x+1-q)^{1/\alpha}-\nu(x-q)^{1/\alpha}+1}{(\nu(x-q)^\frac{1}{\alpha}+1)^{1-\alpha}}
= \sup_{x\geq 0} \frac{\nu(x+1)^{1/\alpha}-\nu x^{1/\alpha}+1}{(\nu x^{1/\alpha}+1)^{1-\alpha}}.
\end{equation*}
This requires the supremum to be finite. Applying the mean value theorem in the nominator and arguing that we take the supremum of a continuous function, we only need to show that the limit as $x$ tends to infinity is finite. This can easily be seen.

Now, suppose $\alpha=0$.
We set $C=e^{\nu}$ and
head for
\begin{equation*}
a_k\geq k+e^{\nu(k-q)}
\end{equation*}
instead of \eqref{InequalityToProve}.
The induction step is
$$
a_{k+1}\ge k+e^{\nu(k-q)}+C(e^{\nu(k-q)})\ge k+1+e^{\nu}e^{\nu(k-q)}
= k+1+e^{\nu(k+1-q)}.
$$
As in the previous case, we obtain a contradiction with
\eqref{borders} for $k$ big enough.
\end{proof}

\subsection{A criterion for Luzin's condition~(N)}
We provide a criterion with a standard proof
for Luzin's condition~(N) that we will apply subsequently.

\begin{proposition}\tlabel{criterion} Assume that $\Omega\subset \field{R}^n$ is open and $\Phi\colon \ropen{0}{\infty}\to [0,\infty]$. We suppose that $f\colon \Omega\to \field{R}^d$ is in $W^{1,1}_\textnormal{loc}(\Omega;\field{R}^d)$ and such that
\begin{equation}\label{integrabilityDifferential}
\int_K \Phi(\norm{Df})\, dy<\infty
\end{equation}
for each compact set $K\subset \Omega$ . Assume that $S\subset \Omega$ is such that for each point $x\in S$ there exists a natural number $L(x)$, two sequences $(r_k(x))_k$ and $(R_k(x))_k$, where the first one converges to zero.
We further stipulate the existence of a sequence $(A_k(x))_k$ of sets $A_k(x)$ such that $A_k(x)\subset B(x,r_k)$, $f(A_k(x))\subset B(f(x),R_k(x))$, and
\begin{align}\label{CriterionRBound}
R_k(x)^n\leq L(x)\int_{A_k(x)} \Phi(\norm{Df})\, dy,\quad k\in \field{N}.
\end{align}
Then $f$ satisfies Luzin's condition~(N) in $S$.
\end{proposition}

\begin{proof}
We may assume that $S$ is relatively compact and
$x\mapsto L(x)$ is bounded on $S$. Let $G\subset\Omega$ be an open set that contains $S$.
Using Vitali covering theorem on the image side, see e.g.\ \cite[Theorem~1.10 (a)]{FalconerGeometry}, we find (infinite or finite, but of
the same length) sequences $(B(x_j,r_j))_j$ of balls in $G$,
$(A_j)_j$ of subsets of $B(x_j,r_j)$,  and $(B(y_j,R_j))_j$ of balls in
$\rn$ such that $y_j=f(x_j)$, the balls
$B(y_j,R_j)$ are pairwise disjoint, cover $f(S)$ up to a set of measure zero,
and for each $j$ there exist $k\in\field{N}$ such
that $r_j=r_k(x_j)$, $R_j=R_k(x_j)$ and $A_j=A_k(x_j)$. Then
$$
\H^{n}(f(S))\le C\sum_{j}R_j^n\le C\sum_j\int_{A_j}\Phi(\norm{Df})\,dy\le C\int_G\Phi(\norm{Df})\,dy.
$$
By specifying the choice of $G$, the integral on the right can be made arbitrarily small, and thus $\H^n(f(S))=0$.
\end{proof}

\subsection{Luzin's condition~(N) for the weakly approximately H\"older continuous part}
Here we take care of the case $\alpha=1$ in \tref{t:main}. As a subcase, we will also need it for $\alpha<1$. In \cite{MalyTheArea}, approximately H\"older continuous mappings were considered. Here, we need as slight generalization.

\begin{definition}[$(\zeta,\vartheta)$\nobreakdash-weakly approximately
H\"older continuous]\tlabel{HoelderDef}
Let $\Omega\subset \field{R}^n$ be open and $(\rho_k)_k$ a
decreasing sequence. For $\zeta,\, \vartheta\in \lopen{0}{1}$, we say that a
mapping $f\colon \Omega\to \field{R}^d$ is
\emph{$(\zeta,\vartheta)$\nobreakdash-weakly approximately H\"older continuous
at $x\in \Omega$ with respect to $(\rho_k)_k$}, if there is a sequence
$(H_k)_k$ of measurable sets $H_k\subset B_k:=B(x,\rho_k)$ such that
\begin{equation*}
\limsup_{k\to \infty} \sup_{y\in H_k} \frac{\norm{f(y)-f(x)}}{\rho_k^\zeta}<\infty,
\end{equation*}
and
\begin{equation*}
\limsup_{k\to \infty} \frac{|B_k\cap H_k|}{|B_k|}\geq \vartheta.
\end{equation*}
We say that $f$ is \emph{weakly approximately H\"older continuous at $x$ with
respect to $(\rho_k)_k$} if there exist $\zeta, \vartheta\in \ropen{0}{1}$ such that $f$ is $(\zeta,\vartheta)$\nobreakdash-weakly approximately H\"older continuous at $x\in \Omega$ with respect to $(\rho_k)_k$.
\end{definition}

\begin{theorem}\tlabel{preponderantlyLuzin}
Suppose that $\Omega\subset \field{R}^n$ is open and $f$ is a mapping in $W^{1,n}_\textnormal{loc}(\Omega,\field{R}^d)$.
Assume that for every point $x\in H$ there exist $\zeta(x),\, \vartheta(x)$ and
a sequence $(\rho_k(x))_k$
converging to zero such that
$f$ is $(\zeta(x),\, \vartheta(x))$ weakly approximately H\"older continuous at $x$ with respect to $(\rho_k(x))_k$.
Then $f$ satisfies Luzin's condition~(N) in $H$.
\end{theorem}

\begin{proof}
For simplicity assume that $\vartheta(x)\ge 1/2$, otherwise we would be forced
to alter the definition of the modulus of continuity.
Consider $\tilde r_k(x)=e^{-k}$ and \mbox{$\tilde
R_k(x)=\omega_f(x,\tilde r_k)$}. Then Lemma \ref{l:hoelder} yields
sets $\tilde A_k\subset B(x,\tilde r_k)$ such that
$$
\tilde R_k-\tilde R_{k+1}\le C\Bigl(\int_{\tilde A_k}|Df|^n\,dy\Bigr)^{1/n}.
$$
In order to apply \tref{criterion} for a subsequence
$(r_k,R_k,A_k)_k$ of $(\tilde r_k,\tilde R_k,\tilde A_k)_k$,
we need the inequality
$$
\tilde R_k\le C(x)(\tilde R_k-\tilde R_{k+1})
$$
to hold for infinitely many $k$. However, if it fails for $k\ge k_0$,
then by iteration we obtain a contradiction with the
weak approximate H\"older continuity for an appropriate choice of
$C(x)$ depending on $\zeta(x)$.
\end{proof}

\subsection{Proof of \stref{t:main}}
We split the set $N$ into two parts. The result then follows from \tref{Holder part} and \tref{Rest Part}.
\begin{proposition}\tlabel{Holder part}
Suppose the same assumptions as in \tref{t:main}. Then $f$ satisfies Luzin's condition~(N) in
\begin{equation*}
S=\bigl\{x\in \Omega:\; \omega_f(x,e^{-k-1})\leq e^{-k}\; \text{for infinitely many $k$}\bigr\}.
\end{equation*}
\end{proposition}

\begin{proof}
Let $N\subset S$ be a set of measure zero.
For $x\in N$, we define $r_k:=e^{-k-1}$ and $R_k:=e^{-k}$.

Our goal is to show that $f$ is weakly approximately H\"older continuous at
$x$ (for $\zeta=1$, $\vartheta=1/2$)
with respect to some subsequence of $(r_k)_k$.  We set
\begin{align*}
H_k&:=\{y\in B(x,r_k):\; \norm{f(y)-f(x)}\leq R_k\}.\\
\end{align*}
Set
$$
J=\{k\in\field{N}\colon \omega_f(x,e^{-k-1})\leq e^{-k} \}.
$$
Then
\begin{equation*}
\limsup_{\substack{k\to
\infty\\k\in J}}
\sup_{y\in H_k} \frac{\norm{f(y)-f(x)}}{r_k}\leq \limsup_{\substack{k\to
\infty\\k\in J}}  \frac{e^{-k}}{e^{-k-1}}<\infty
\end{equation*}
and
$$
\measure{H_k}\geq \frac{1}{2}\measure{B_k},\qquad k\in J.
$$
Since $J$ is infinite, we may apply \tref{preponderantlyLuzin}.
\end{proof}

\begin{proposition}\tlabel{Rest Part}
Assume that the assumptions in \tref{t:main} hold.
Then $f$ satisfies Luzin's condition~(N) in the set
$$
S=\bigl\{x\in \Omega:\; \omega(e^{-k-1})\geq \medmod(x,e^{-k-1})> e^{-k}\; \text{for almost all $k$}\bigr\}.
$$
where
\begin{equation*}
\omega(r)=
\begin{cases}
 \exp\bigl(-\mu\log^{\alpha}\bigl(1/r\bigr)\bigr) ,&\alpha>0,\\
\log^{-\mu}(1/r),&\alpha=0.
\end{cases}
\end{equation*}
\end{proposition}

\begin{proof}
The goal is to apply \tref{criterion}. Assume $N\subset S$ is a set of measure
zero. We may assume that $\mu$ is constant on $N$. If $x\in N$, we choose $q\in
\field{N}$ such that $e^{-q}<R$ and such that $\omega_f(x,e^{-k-1})>e^{-k}$ for
all $k\geq q$.  We define
$$
r_k:=r_k(x):=\inf\{t>0: \omega_f(x,t) \ge e^{-k}\},\qquad k=q,\,q{+}1,\,q{+}2,\dots.
$$
It is clear that $e^{-k}$
is a median of $\norm{f-f(x)}$ in $B(x,r_k)$.
Since Sobolev functions have unique medians, we have
$$
R_k:=R_k(x):=\omega_f(x,r_k)=e^{-k}.
$$
As $x\in S$, using Lemma \ref{l:hoelder}, we obtain
\begin{equation}\label{UsingMedianEquation}
R_k^n=e^{-kn}\le C\bigl(1+\log(r_k/r_{k+1})\bigr)^{n-1}
\int_{A_k}\norm{Df}^n\,dy,
\end{equation}
where
\begin{equation*}
A_k:=A_k(x):=\{y\in B(x,r_k):\;|f(y)-f(x)|\le R_k(x)\}.
\end{equation*}
From now on, $C$ may vary from line to line but does not depend on $k$.
Our goal is to exploit the full integrability of $\norm{Df}$, i.e.\ we want to prove that
\begin{equation}\label{AlaRadoReichelderfer}
R_k^n\leq \frac{L}{n^\lambda}\int_{A_k} \norm{Df}^n\log^\lambda(e+\norm{Df})\, dy,
\end{equation}
for infinitely many $k$, where $L$ is independent of $k$. Towards applying Jensen's inequality with $\Phi(t)=t\log^\lambda(e+t)\sim
t\log^\lambda(e+t^{1/n})$, we rewrite \eqref{UsingMedianEquation} as
\begin{equation*}
\frac{R_k^n}{Cr_k^n(1+\log(r_k/r_{k+1}))^{n-1}}\leq
\dashint_{B(x,r_k)}\norm{Df}^n\chi_{A_k}\,dy.
\end{equation*}
We conclude that
\begin{equation*}
\frac{R_k^n}{Cr_k^n(1+\log(r_k/r_{k+1}))^{n-1}}\log^\lambda\Bigl(e+\frac{R_k^{n}}{Cr_k^n(1+\log(r_k/r_{k+1}))^{n-1}}\Bigr)
\leq \frac{1}{r_k^n}\int_{A_k} \Phi(\norm{Df}^n)\, dy.
\end{equation*}
Hence, it suffices to show that
\begin{equation}\label{super}
\limsup_{k\to \infty} \frac{\log^\lambda\bigl(e+\frac{e^{-kn}}{Cr_k^n(1+\log(r_k/r_{k+1}))^{n-1}}\bigr)}{(1+\log(r_k/r_{k+1}))^{n-1}}>0.
\end{equation}
To this end, we first consider the inequality
\begin{equation}\label{Bridging inequality}
(1+\log(r_k/r_{k+1}))^{n-1}\leq C\log^\lambda\Bigl(e+\frac{e^{-k}}{r_k}\Bigr).
\end{equation}
In fact, setting $a_j=\log(1/r_j)$, we denote the set of all indices that
satisfy inequality~\eqref{recurrent} by $J$.
Then $k\in J$
satisfy
\begin{equation*}
\log(r_k/r_{k+1})\leq C\Bigl(\log \frac{e^{-k}}{r_k}\Bigr)^{1-\alpha},
\end{equation*}
and \eqref{Bridging inequality} follows.
To prove that $J$ is infinite,
it suffices to verify the assumptions of
\tref{l:elementary}.

We see that $r_j\le e^{-j-1}$.
Obviously, $(a_j)_j$ forms a nondecreasing sequence and $a_j\ge j+1$.

If $\alpha>0$, we estimate
$$
     e^{-j}=\omega_f(x,r_j)\le \exp(-\mu\log^{\alpha}(1/r_j))=\exp(-\mu a_j^{\alpha}),
$$
so that
$$
j\ge \mu a_j^{\alpha},\qquad a_j\le (j/\mu)^{1/\alpha}.
$$

If $\alpha=0$,
we estimate
$$
     e^{-j}=\omega_f(x,r_j)\le \log^{-\mu}(1/r_j)=a_j^{-\mu},
$$
so that
$$
a_j\le e^{j/\mu}.
$$
The application of \tref{l:elementary} is permitted, and we see that \eqref{Bridging inequality} holds for infinitely many $k$,
where, in the notation of \tref{l:elementary}, $C=C(\mu^{-1/\alpha},\alpha)$
if $\alpha>0$ and $C=C(1/\mu,\alpha)$ if $\alpha=0$.

Inequality~\eqref{Bridging inequality} shows that
$(1+\log(r_k/r_{k+1}))^{n-1}\le C(e^{-k}/r_k)^{n-1}$ for all $k\in J$;
Indeed, for $s\ge e$ we have
$$
\log(e+s)=\int_1^{e+s}\frac{dt}{t}\le
\int_1^{s}\,dt+\int_{s}^{e+s}\frac{dt}{s}\le s-1+\frac es\le s-1+1=s,
$$
and $s:=\frac{e^{-k}}{r_k}\ge e$. Thus,
$$
\log^{\lambda}\Bigl(e+\frac{e^{-k}}{r_k}\Bigr)\le
\Bigl(\frac{e^{-k}}{r_k}\Bigr)^{\lambda}\le
\Bigl(\frac{e^{-k}}{r_k}\Bigr)^{n-1}.
$$
Hence there exists $C'>0$ such that for every $k\in J$ we have
$$
\log^\lambda\Bigl(e+\frac{e^{-kn}}{Cr_k^n(1+\log(r_k/r_{k+1}))^{n-1}}\Bigr)\ge
\log^\lambda\Bigl(e+\frac{e^{-k}}{C'r_k}\Bigr).
$$
Using \eqref{Bridging inequality} once more, \eqref{super} reduces to
$$
\limsup_{\substack{k\to 0\\k\in J}}
\frac{\log^{\lambda}\Bigl(e+\frac{e^{-k}}{C'r_k}\Bigr)}
{\log^{\lambda}\Bigl(e+\frac{e^{-k}}{r_k}\Bigr)} >0,
$$
which is easily verified.

Thus we have found a constant $L$ depending only on $x$, $\mu$, $\alpha$, $n$, and $f$ such that \eqref{AlaRadoReichelderfer} is true. We apply \tref{criterion} to finish.
\end{proof}

\section{Example}\label{Section Example}
Here, we construct an example as required in \tref{ExampleModulusIntro}.

We split the construction and the verification of the properties of the example over the next subsections. First, we note that it suffices to consider the case $d=n$; if $d>n$ and we have an example $f\colon \field{R}^n\to \field{R}^n$, we define $F\colon \field{R}^n\to \field{R}^d$ as $F(x)=(f_1(x),\ldots, f_n(x),0,\ldots,0)$.

\subsection{Definition of the mapping}
Let us start with the construction. We denote
$\beta=1/\alpha>1$ and set
\begin{align*}
r_j&:=2^n\exp(-(j+1)^\beta),\\
R_j&:=\exp(-j^\beta).
\end{align*}
Notice that $r_j/R_{j+1}=2^n$, so that each ball of radius $r_j$
contains at least $2^n$ pairwise disjoint balls of radius $R_{j+1}$.
Let us choose a natural number $j_0\ge 2$ such that
\begin{equation}\label{the choice}
(\beta-1)j_0^{\beta-1}>n\log2,\quad \sqrt n<2^{j_0-1}
\quad\text{and}\quad 2eR_{j_0}<1.
\end{equation}
Using the mean value theorem,  for $j > j_0$ we obtain
\begin{equation}\label{estimate c}
\aligned
\log\Bigl(\frac{R_j}{r_j}\Bigr)&
=\log\Bigl(\frac{\exp(-j^\beta)}{2^n\exp(-(j+1)^\beta)}\Bigr)
\\&
=
(j+1)^\beta-j^\beta-n\log 2\geq \beta j^{\beta-1}-n\log 2\geq j^{\beta-1},
\endaligned
\end{equation}
in particular $r_j<R_j$.
In the first generation, we choose $2^n$  balls
$B(a_{i,j_0+1},R_{j_0+1})$
of radius $R_{j_0+1}$ and the same number of concentric
balls with radius
$r_{j_0+1}$. In each subsequent generation with $j> j_0+1$, we choose in every ball of radius
$r_{j-1}$ of the previous generation
$2^n$ pairwise disjoint open balls of radius $R_{j}$
and in each of the new
balls a concentric closed  ball with radius $r_{j}$.
We denote by $a_{i,j}$ the
centers, $i=1,\dots,2^{(j-j_0)n}$.

The function
$$
\tilde\eta_j(r)=
\frac{\log R_j-\log r}{\log R_j-\log r_j}
$$
attains the values $\tilde\eta_j(R_j)=0$, $\tilde\eta_j(r_j)=1$.
By truncation and smoothing we find a smooth function $\eta_j:\er\to\er$ such that
$$
\aligned
&\eta_j(r)=0, && r\ge R_j,\\
&\eta_j(r)=1, && r\le r_j,\\
&0\le -r\eta_j'(r)\le 2\log^{-1}\Bigl(\frac{R_j}{r_j}\Bigr),\qquad && r_j<r<R_j.
\endaligned
$$
By \eqref{estimate c} we have
\begin{equation}\label{estimate of c}
0\le -\eta_j'(r)\le 2j^{1-\beta}\frac{1}{r},\qquad j> j_0.
\end{equation}
We define
\begin{equation*}
\Psi_{ij}(x):=\eta_j(\norm{a_{i,j}-x}),\qquad i=1,\dots,2^{(j-j_0)n}
\end{equation*}
and given a bijection $\sigma\colon \{1,\ldots, 2^n\}\to \{-1,1\}^{n}$, we set $v_k=\sigma(k)$ if $1\leq k\leq 2^n$ and $v_k:=v_{k-2^n}$ for $k>2^n$. Further, we define
\begin{equation*}
w_{ij}:=2^{-j}v_i,\qquad
i=1,\dots,2^{(j-j_0)n}.
\end{equation*}
Finally, we let
\begin{equation*}
f:=\sum_{j=j_0+1}^{\infty}f_j
\quad\text{with}\quad
f_j=
\sum_{i=1}^{2^{(j-j_0)n}} w_{ij}\Psi_{ij}(x).
\end{equation*}
The
mapping $f$ will serve as the required example.

\subsection{The mapping blows up a set of measure zero}
Let
$$
N_j:=\bigcup_{i=1}^{2^{(j-j_0)n}} \overline{B}(a_{i,j},r_j),\qquad N:=\bigcap_{j=j_0+1}^{\infty} N_j.
$$
Then
$$
|N_j|\le C2^{nj}r_j^n\le  C2^{nj}\exp(-n(j+1)^{\beta})\to0\quad \text{as }j\to\infty
$$
and thus $|N|=0$.
On the other hand, each dyadic cube intersecting  $[-2^{-j_0},2^{-j_0}]^n$
also intersects $f(N)$. The continuity of $f$, which we will verify in Subsection~\ref{VerifyingContinuity}, together with the fact that $N$ is compact,  now shows that $f(N)$ contains $[-2^{-j_0},2^{-j_0}]^n$ and hence has positive measure.

\subsection{The integrability of the derivative}
We write $A_{i,j}$ for the closed annulus with center $a_{i,j}$ and inner and outer radii $r_j$ and $R_j$, respectively, and
set
$$
A_j=\bigcup_{i=1}^{2^{(j-j_0)n}}A_{i,j}
$$
Then $Df_j=0$ outside $A_j$. Since the annuli $A_{i,j}$ are pairwise disjoint and $|w_{ij}|/\sqrt n=2^{-j}$ we have
\begin{equation}\label{WeakDerivative}
|Df(x)|=
|Df_j(x)|=\bigl|w_{ij}D\Psi_{ij}(x)\bigr|\le \sqrt n\;2^{-j}
|\eta_j'(|x-a_{i,j}|)|,\qquad x\in A_{i,j}.
\end{equation}
Using \eqref{estimate of c} we obtain
$$
\aligned
\int_{\rn}|Df_j|\, dy&\le C2^{nj}\int_{r_j}^{R_j}r^{n-1}2^{-j}|\eta_j'(r)|\,dr
\le C2^{nj-j}j^{1-\beta}\int_{r_j}^{R_j}r^{n-2}\,dr
\\&
\le  C2^{nj-j}j^{1-\beta}R_j^{n-1}=C2^{nj-j}j^{1-\beta}\exp(-(n-1)j^{\beta}).
\endaligned
$$
It follows that the partial sums form a fundamental sequence in $W^{1,1}$
and the total sum is a Sobolev function. Since $|N|=0$, we have
$$
\int_{\rn}|Df|^n\log^{\lambda}(e+|Df|)\, dy=
\sum_{j=j_0+1}^{\infty}I_j,
$$
where
\begin{align*}
I_j&=\int_{A_j}|Df|^n\log^{\lambda}(e+|Df|)\, dy\\
&\le C2^{nj}\;2^{-nj}\int_{r_j}^{R_j}r^{n-1}|\eta_j'(r)|^n\log^{\lambda}(e+\sqrt n\;2^{-j}|\eta_j'(r)|)\,dr.
\end{align*}
By \eqref{the choice} and \eqref{estimate of c}, we have
$$
e+\sqrt n\;2^{-j}|\eta_j'(r)|\le\frac1{2R_{j_0}}+ \frac1{2r}\le \frac1r,\qquad r_j<r<R_j.
$$
We estimate
$$
\aligned
I_j&\le Cj^{n(1-\beta)}\int_{r_j}^{R_j}\log^{\lambda}\Bigl(\frac 1r\Bigr)\;\frac{dr}{r}
\le Cj^{n(1-\beta)}\Bigl(\log^{\lambda+1}\Bigl(\frac1{r_j}\Bigr)
-\log^{\lambda+1}\Bigl(\frac1{R_j}\Bigr)\Bigr)
\\&\le Cj^{n(1-\beta)}\Bigl((j+1)^{\beta(\lambda+1)}-j^{\beta(\lambda+1)}\Bigr)
\le Cj^{\lambda\beta+\beta-1+n-n\beta}.
\endaligned
$$
We conclude that the integrability of the derivative is as required.

\subsection{Modulus of continuity}\label{VerifyingContinuity}
We seek for an estimate of the modulus of continuity of
$\tilde f=f_{j_0}+\dots+f_J$ which does not depend on $J$. The same estimate
then holds for $f$. Without loss of generality, choose
$x,y\in\rn\setminus N$
such that $\tilde f(x)\ne0\ne \tilde f(y)$.
Then there exist chains of balls,
$$
\aligned
&\ball(a_{j_0+1}(x),R_{j_0+1})\supset \ball(a_{j_0+2}(x),R_{j_0+2})
\supset\dots\supset  \ball(a_{j_x}(x),R_{j_x})\ni x,\\
&\ball(a_{j_0+1}(y),R_{j_0+1})\supset \ball(a_{j_0+2}(y),R_{j_0+2})
\supset\dots\supset  \ball(a_{j_y}(y),R_{j_y})\ni y,
\endaligned
$$
such that $a_{j}(x),\,a_{j}(y)\in \{a_{i,j}\colon i=1,\dots,2^{(j-j_0)n} \}$
and
$$
\aligned
&x\notin \bigcup_{j>j_x} \bigcup_{i=1}^{2^{(j-j_0)n}}\ball(a_{i,j},R_{j}),
&y\notin \bigcup_{j>j_y} \bigcup_{i=1}^{2^{(j-j_0)n}}\ball(a_{i,j},R_{j}).
\endaligned
$$
We denote by $A_j(x)$ and $A_j(y)$ the corresponding annuli
$\ball(a_j(x),R_j)\setminus B(a_j(x),r_j)$ and
$\ball(a_j(y),R_j)\setminus B(a_j(y),r_j)$, respectively.
We may assume that $j_x\le j_y$. Let $L$ be the line segment connecting $x$ and $y$. Our next step is to find $x',y'\in L$ such that
\begin{equation}\label{primed}
|\tilde f(y)-\tilde f(x)|\le C|\tilde f(y')-\tilde f(x')|,
\end{equation}
and $x'$ and $y'$ belong both to the same annulus $A_{i,j}$. This is possible
if $\tilde f(y)\ne\tilde f(x)$, but otherwise there is nothing to be estimated.

We distinguish several cases.

{\sc Case 1.} There exists $a\in\rn$ such that $a_{j_x}(x)=a_{j_x}(y)=a$
and $j_y=j_x$. We may assume that
$|x-a|\ge |y-a|$. Then $x\in A_{j_x}(x)$, as otherwise we would get
$\tilde f(x)=\tilde f(y)$. We
set $x'=x$ and find $y'\in L\cap A_{j_x}(x)$ such that
$\tilde f(y')=\tilde f(y)$. We have $|\tilde f(x')-\tilde f(y')|=|\tilde f(x)-\tilde f(y)|$.

{\sc Case 2.} There exist $a,b\in\rn$ such that $a_{j_x}(x)=a_{j_x}(y)=a$,
$a_{j_y}(y)=b$
and $j_y=j_x+1$. Then we find $z\in L\cap \partial B(b,R_{j_y})$.
If $|\tilde f(y)-\tilde f(z)|\ge |\tilde f(x)-\tilde f(z)|$, we find
$y'\in L\cap A_{j_y}(y)$
such that $\tilde{f}(y')=\tilde{f}(y)$ and
set $x'=z$.  If $|\tilde f(y)-\tilde f(z)|< |\tilde f(x)-\tilde f(z)|$,
we find $y'\in L\cap \partial B(a,r_{j_x})$
and set $x'=x$. We obtain
$|\tilde f(x')-\tilde f(y')|\ge\tfrac12|\tilde f(x)-\tilde f(y)|$.

{\sc Case 3.} There exist $a,b\in\rn$ such that $a_{j_x}(x)=a_{j_x}(y)=a$,
$a_{j_x+1}(y)=b$
and $j_y>j_x+1$. We find $x'\in L\cap \partial B(b,R_{j_x+1})$ and
$y'\in L\cap \partial B(b,r_{j_x+1})$.
Then
$$
|\tilde f(y)-\tilde f(x)|\le \sqrt n\;(2^{-j_x}+2^{-j_x-1}+\cdots)\le 4\sqrt n\;
2^{-j_x-1}
\le 4|\tilde f(y')-\tilde f(x')|.
$$

{\sc Case 4.} There exists $j_1\le j_x$
such that $a_{j_1}(x)\ne a_{j_1}(y)$, and $a_{j}(x)= a_{j}(y)$
whenever $j_0<j<j_1$. Then we find $x''\in L\cap \partial B(a_{j}(x),R_j)$,
$y''\in L\cap \partial B(a_{j}(y),R_j)$ and estimate
$$
|\tilde f(y)-\tilde f(x)|\le|\tilde f(y)- \tilde f(y'')|+|\tilde f(x'')-\tilde f(x)|.
$$
The relation of $x$ and $x''$ and of $y$ and $y''$, respectively, is one of these described in the preceding steps. Hence, we find $x'$, $y'\in L$
such that  $x'$ and $y'$ belong both to the same annulus $A_{i,j}$ and
$$
|\tilde f(y)-\tilde f(x)|\le 8|\tilde f(y')- \tilde f(x')|.
$$

Therefore, in either case we obtain \eqref{primed} with $C=8$.
Letting $J\to\infty$, we observe that
$$
|f(y)-f(x)|\le 8 \sup_{i}\sup_{j}\{|f(\bar y)- f(\bar x)|\colon \bar x,\bar y\in A_{i,j},\ |\bar y-\bar x|\le |y-x|\}.
$$
Since $\omega$ is increasing,
we can reduce the estimate to the case that $x$ and $y$ belong
both to the same annulus $A_{i,j}$.
We have
\begin{equation*}
\begin{split}
\norm{f(x)-f(y)}&=\norm{f_j(x)-f_j(y)}\\
&=\norm{w_{ij}\Psi_{ij}(x)-w_{ij}\Psi_{ij}(y)}=2^{-j}\sqrt n\;
\abs{\eta_j(\norm{a_{i,j}-x})-\eta_j(\norm{a_{i,j}-y})}.
\end{split}
\end{equation*}
We assume that $\norm{a_{i,j}-x}\leq \norm{a_{i,j}-y}$ and set $r:=\norm{a_{i,j}-x}$,
$h=\norm{a_{i,j}-y}-r$. Then $0\le h\le \min\{|x-y|,R_j-r_j\}$.
Thus
\begin{equation*}
\frac{\norm{f(x)-f(y)}}{\sqrt n}\leq 2^{-j}\int_r^{r+h}\abs{\eta'(t)}\,dt
\le \frac{2j^{1-\beta}}{2^j}\log\Bigl(1+\frac{h}{r}\Bigr)
\leq \frac{j^{1-\beta}}{2^{j-1}}\log\Bigl(1+\frac{h}{r_j}\Bigr).
\end{equation*}
Set
$$
s=\log^{\frac 1\beta}\bigl(\tfrac{1}{h}\bigr).
$$
Then
$$
e^{-j^{\beta}}=R_j\ge h=e^{-s^{\beta}},
$$
and thus $s\ge j$. We will distinguish two cases.

{\sc Case 1.} $j\le s\le j+1$. Then
$$
\aligned
&j^{1-\beta}2^{1-j}\log\Bigl(1+\frac{e^{-s^{\beta}}}
{2^ne^{-(j+1)^{\beta}}}\Bigr)
\\&\quad
\le
j^{1-\beta}2^{1-j}\log\Bigl(2\frac{e^{-s^{\beta}}}
{e^{-(j+1)^{\beta}}}\Bigr)\le
j^{1-\beta}2^{1-j}(\log 2 +(j+1)^{\beta}-j^{\beta})
\\&\quad
\le
j^{1-\beta}2^{1-j}(\log 2 +\beta (j+1)^{\beta-1})\le 2^{1-j}(\log 2+\beta2^{\beta-1})\le C2^{-s}.
\endaligned
$$

{\sc Case 2.} $s>j+1$. Then
$$
s^{\beta}-(j+1)^{\beta}\ge \beta(s-j-1)(j+1)^{\beta-1}\ge (s-j-1).
$$
Hence
$$
\aligned
&j^{1-\beta}2^{1-j}\log\Bigl(1+\frac{e^{-s^{\beta}}}
{2^ne^{-(j+1)^{\beta}}}\Bigr)
\\&\quad
\le
2^{1-j-n}\frac{e^{-s^{\beta}}}
{e^{-(j+1)^{\beta}}}\le 2^{1-j-n}e^{j+1-s}\le 2^{1-j-n}2^{j+1-s}\le C2^{-s}.
\endaligned
$$

In both cases we have
$$
| f(y)-f(x)|\le C2^{-s}=C2^{-\log^{\frac 1\beta}(\frac 1h)}
\le C2^{-\log^{\frac 1\beta}(\frac 1{|y-x|})},
$$
which is the required modulus of continuity.

\subsection{The case \texorpdfstring{$\lambda=n-1$}{lambda=n-1}}
We may proceed as in the case $0\le \lambda<n-1$ and
construct
a mapping $f\in W^{1,1}(\field{R}^n;\field{R}^n)$
such that
\begin{equation*}
\int_{\field{R}^n} \norm{Df}^n\log^{n-1}(e+\norm{Df})\,dy<\infty,
\end{equation*}
$f$ has modulus
of continuity no worse than
\begin{equation*}
\Psi(t)=C\exp(-\log2(\log\log(1/t))^{\frac{1}{\beta}})
\end{equation*}
and violates Luzin's condition~(N).

In this case, we fix $\beta >\frac {n}{n-1}$ and set
\begin{align*}
r_j&:=2^n\exp(-\exp(j+1)^\beta),\\
R_j&:=\exp(-\exp j^\beta).
\end{align*}
The only other substantial change to the construction that we gave above is that
we set
$$
\tilde\eta_j(r)=
\frac{\log \log (1/r)-\log \log(1/R_j)}{\log \log(1/r_j)-\log \log(1/R_j)}.
$$
We leave the details to the reader.

\section{Size of the exceptional set}\label{Sec exceptional set}
Marcus and Mizel have shown in \cite{MM73} that continuous mappings in $W^{1,p}(\Omega;\field{R}^n)$ satisfy Luzin's condition~(N) for $p>n$. From \cite{mal_lusins_1995}, we know that given an $n$\nobreakdash-quasicontinuous mapping $f$ in $W^{1,n}(\Omega;\field{R}^n)$, we find a zero-dimensional set $E$ (depending on $f$) such that $f$ satisfies Luzin's condition~(N) in $\Omega\setminus E$. Here, we give an upper bound for the size of the exceptional set when we have an additional logarithmic term in the integrability condition for $\norm{Df}$.

We will use the following result to estimate the size of the set where $f$ fails to satisfy Luzin's condition.

\begin{lemma}\tlabel{Measure bound}
Let $v\colon \Omega\to [0,\infty)$ be an integrable function and $\varphi$ a gauge function.
Let $E\subset\Omega$.
Assume that
\begin{equation*}
\limsup_{r\to0_+}( \varphi(2r))^{-1}\int_{B(x,r)} v\, dy>1 ,\qquad x\in E.
\end{equation*}
Then there exists a constant $C=C(n)$ such that
\begin{equation*}
\mathcal{H}^\varphi(E)\leq C\int_U v\, dy
\end{equation*}
for each open set $U\subset \Omega$ containing $E$.\\
Moreover, if $\measure{E}=0$, then $\mathcal{H}^\varphi(E)=0$.
\end{lemma}

\begin{proof}
Let $U$ be an open set containing $E$ and fix $\delta>0$. Choose for each $x\in E$ a radius $r_x<\delta/2$ such that $B(x,2r_x)\subset U$ and
\begin{equation*}
\int_{B(x,r_x)} v\, dy\geq \varphi(2r_x).
\end{equation*}
By Besicovitch's Theorem, see for example Theorem~1.5.2 in \cite{EG92}, there exist a constant $N=N(n)$ and subcollections $\mathcal{G}_1,\ldots, \mathcal{G}_m$, $m\leq N$, such that each of these collections consists of disjoint balls, and the union of all collections covers $E$. Then
\begin{equation*}
\mathcal{H}^\varphi_\delta(E)\leq \sum_{i=1}^{N}\sum_{B(x,r_x)\in \mathcal{G}_i}\varphi(2r_x)\leq \sum_{i=1}^{N}\sum_{B(x,r_x)\in \mathcal{G}_i} \int_{B(x,r_x)}v\, dy \leq N\int_U v\, dy.
\end{equation*}
Since $\delta>0$ was arbitrary, the first claim follows.

Assume now that $\measure{E}=0$. Given $\varepsilon>0$, by the absolute continuity of the integral and the fact that $v$ is integrable, we can choose the open set $U$ so small that
\begin{equation*}
N\int_U v\, dy<\varepsilon.
\end{equation*}
This implies $\mathcal{H}^\varphi(E)<\varepsilon$ and letting $\varepsilon$ tend to zero, we obtain the claim.
\end{proof}

Recall that $f\in L_{\loc}^1(\Omega;\rd)$
is said to be a \textit{precise representative} if for each $x\in\Omega$ and
$c\in\rd$ we have
$$
\lim_{r\to 0_+}\dashint_{B(x,r)}|f-c|\,dy=0 \implies c=f(x).
$$
To investigate moduli of continuity of precise representatives, we need a joint
estimate for integral means and medians.

For the rest of this section, we fix $0\le \lambda\leq n-1$
and $\mu>0$,
write $\alpha=1-\frac{\lambda}{n-1}$ and set
$$
\aligned
\varphi(r)&=
\begin{cases}
\exp\bigl(-(n+1)\mu\log^{\alpha}(\tfrac{1}{r})\bigr),&\alpha>0,\\
2^{-n}\mu^n\log^{-n(1+\mu)}(\tfrac{1}{r}),&\alpha=0,
\end{cases}\\
\psi(r)&=
\begin{cases}
\exp(-\mu\log^{\alpha}(\tfrac{1}{r})),&\alpha>0,\\
\log^{-\mu}(\tfrac{1}{r})),&\alpha=0.
\end{cases}
\endaligned
$$

\begin{lemma}\tlabel{Sum of moduli}
There exists a radius $R_0=R_0(\alpha,\mu)$ such that
$$
\psi(2r)-\psi(r)\ge (\varphi(r))^{1/n},\qquad 0<r<R_0.
$$
\end{lemma}

\begin{proof} Assume first that $\alpha>0$. We have
$$
\aligned
\psi(2r)-\psi(r)&=
\exp\bigl(-\mu\log^{\alpha}(\tfrac{1}{2r})\bigr)-
\exp\bigl(-\mu\log^{\alpha}(\tfrac{1}{r})\bigr)
\\&
\ge
\mu\exp\bigl(-\mu\log^{\alpha}(\tfrac{1}{r})\bigr)
\Bigl(\log^{\alpha}(\tfrac{1}{r})-\log^{\alpha}(\tfrac{1}{2r})\Bigr).
\endaligned
$$
Now, it is enough to find $R_0$ such that
$$
\mu\Bigl(\log^{\alpha}(\tfrac{1}{r})-\log^{\alpha}(\tfrac{1}{2r})\Bigr)
\ge\exp\bigl(-\tfrac{\mu}{n}\log^{\alpha}(\tfrac{1}{r})\bigr),
\qquad 0<r<R_0.
$$
If $\alpha=0$, we estimate
$$
\aligned
\psi(2r)-\psi(r)&=
\log^{-\mu}(\tfrac{1}{r})-
\log^{-\mu}(\tfrac{1}{2r})
\\&
\ge
\mu\log^{-\mu-1}(\tfrac{1}{r})
\bigl(
\log(\tfrac{1}{r})
-\log(\tfrac{1}{2r})
\bigr)
=\mu\log2\log^{-\mu-1}(\tfrac{1}{r}).
\endaligned
$$
\end{proof}

\begin{lemma}\tlabel{Lebesgue point}
Let $f\in W^{1,n}(\Omega;\field{R}^d)$ and $x\in\Omega$.
Suppose that
\begin{equation*}
\limsup_{r\to0_+}\frac{\int_{B(x,r)}|Df|^n\, dy}{\varphi(2r)}\le 1.
\end{equation*}
Then there is a constant $C=C(\alpha,\gamma,n)$ and $R=R(x)>0$ such that for each
$0<t<r$ we have
\begin{equation}\label{estimate fbs}
\norm{f_{B(x,r)}-f_{B(x,t)}}
+|\omega_{f}(x,r)-\omega_{f}(x,t)|\leq C\psi(r).
\end{equation}
\end{lemma}

\begin{proof}
Let $R_0$ be the radius from \tref{Sum of moduli}. Choose $0<R<R_0$
such that
$B(x,R)\subset\Omega$ and
\begin{equation}\label{added maximum bound}
\int_{B(x,r)}|Df|^n\, dy\le 2\varphi(2r)\le C\varphi(r/2),\qquad 0<r<R.
\end{equation}
If $\varphi$ is concave on a right neighborhood of $0$, then we may choose $C=8$. This fails (and calls for a different constant) only in the case $\alpha=1$, $(n+1)\mu>1$.

First consider the case $0<t<r\le 2t<R$. Then
$$
\norm{f_{B(x,r)}-f_{B(x,t)}}
+|\omega_{f}(x,r)-\omega_{f}(x,t)|
\le C
\Bigl(\int_{B(x,r)} \norm{Df}^n\, dy\Bigr)^{\frac{1}{n}}.
$$
The estimate of $\norm{f_{B(x,r)}-f_{B(x,t)}}$ is a standard use of the Poincar\'e inequality; for the estimate of $|\omega_{f}(x,r)-\omega_{f}(x,t)|$
we apply the Poincar\'e inequality
(to $u=|f-f(x)|$) like in the proof of Lemma \ref{l:riesz}.
We continue by applying \eqref{added maximum bound} and \tref{Sum of moduli}
and obtain
\begin{equation}\label{withdiff}
\aligned
&\norm{f_{B(x,r)}-f_{B(x,t)}}
+|\omega_{f}(x,r)-\omega_{f}(x,t)|
\leq
C(\varphi(r/2))^{\frac1n}
\\&\quad
\le C(\psi(r)-\psi(r/2)).
\endaligned
\end{equation}
In the general case $0<t<r<R$,
we find $j\in\en$ such that $2^{j-1}\le r/t<2^{j}$ and
consider a chain of balls $B(0,r_k)$,
$k=0,\dots,j$, where $r_0=r$ and
$r_k=t2^{j-k}$ for $k=1,\dots,j$.
Then we apply \eqref{withdiff} to each ball in the chain. Summing over
$k$ and using the cancellation effect of \eqref{withdiff} we obtain
$$
\norm{f_{B(x,r)}-f_{B(x,t)}}
+|\omega_{f}(x,r)-\omega_{f}(x,t)|\le C\psi(r).
$$
\end{proof}

\begin{lemma}\tlabel{modulus bounds} Let $f\in W^{1,n}(\Omega;\field{R}^d)$ be a precise representative and $x\in\Omega$.
Suppose that
\begin{equation*}
\limsup_{r\to0_+}\frac{\int_{B(x,r)}|Df|^n\, dy}{\varphi(2r)}\le 1.
\end{equation*}
Then there exists a constant $C=C(\alpha,\gamma)$ and $R=R(x)>0$ such that
\begin{equation*}
\omega_f(x,r)\leq C\psi(r),\qquad 0<r<R.
\end{equation*}
\end{lemma}

\begin{proof}
By \tref{Lebesgue point}, there exists $R>0$ such that
$$
\norm{f_{B(x,r)}-f_{B(x,t)}}\le C\psi(r),\qquad 0<t<r<R.
$$
This guarantees the existence of the limit
$$
c=\lim_{r\to0_+}f_{B(x,r)}.
$$
By the Poincar\'e inequality, we have
$$
\aligned
\dashint_{B(x,r)}|f-c|\,dy&\le |f_{B(x,r)}-c|+\dashint_{B(x,r)}|f-f_{B(x,r)}|\, dy
\\&
\le |f_{B(x,r)}-c|+
C\Bigl(\int_{B(x,r)}|Df|^n\, dy\Bigr)^{1/n}\to 0\quad \text{as }r\to 0_+.
\endaligned
$$
Thus $c=f(x)$ and $x$ is a Lebesgue point for $f$. It easily follows that
$$
\lim_{t\to 0_+} \omega_f(x,t)=0.
$$
Now, using \eqref{estimate fbs} again we have
$$
|\omega_f(x,r)-\omega_f(x,t)|\le C\psi(r),\qquad 0<t<r<R.
$$
We pass to the limit for $t\to0_+$ and obtain the required estimate.
\end{proof}

\begin{proof}[Proof of \tref{Size of exceptional set}]
Without loss of generality, we may assume that $\gamma<n$. Set
$$
E=\biggl\{
x\in\Omega:\;\limsup_{r\to0_+}\frac{\int_{B(x,r)}|Df|^n\, dy}{\varphi(2r)}>1\biggr\}.
$$
Then $E$ does not contain any Lebesgue point of $|Df|^n$ and thus $|E|=0$.
\tref{Measure bound} gives
\begin{equation*}
\mathcal{H}^\varphi(E)=0.
\end{equation*}
By \tref{modulus bounds}, for each $x\in\Omega\setminus E$ we have
$$
\omega_f(x,r)\le C\psi(r),\qquad 0<r<R(x).
$$
Hence we may apply \tref{t:main} to obtain the claim.
\end{proof}

\section{Regularity of the blown up set}\label{Section Regularity}
As compact, perfect, totally disconnected sets are exactly the sets homeomorphic to the ternary  Cantor set, cf.\ Corollary~2-98 in \cite{HockingYoung}, we refer to them as \emph{Cantor sets}. In \tref{ExampleModulusIntro}, a Cantor set gets mapped onto a set of positive measure. Our next result shows all examples must exhibit such a behavior.

\begin{theorem}\tlabel{compact set} Let $\Omega\subset \field{R}^n$ be open and $f\colon \Omega\to \field{R}^d$ be continuous. If for some gauge function $\varphi$
there is a set $N\subset \Omega$ with $\mathcal{H}^\varphi(N)=0$ and $\mathcal{H}^n(f(N))>0$, then there is a Cantor set $K$ such that $\mathcal{H}^\varphi(K)=0$ and $\mathcal{H}^n(f(K))>0$.
\end{theorem}

First, we need a result about the image of an intersection.
\begin{lemma}\tlabel{intersection of compacts}
Let $A\subset \field{R}^n$ and $f\colon A\to \field{R}^d$ be a continuous mapping. We suppose that $K_1\supset K_2\supset \cdots$ is a sequence of nested compact subsets of $A$. Then
\begin{equation*}
f\Bigl(\bigcap_{m=1}^\infty K_m\Bigr)=\bigcap_{m=1}^\infty f(K_m).
\end{equation*}
\end{lemma}

\begin{proof} This is a standard application of compactness, see e.g.\
Exercises 4.28 and 4.29vi in \cite{Kechris}.
\end{proof}

\begin{proof}[Proof of \tref{compact set}]
Since Hausdorff measures are Borel regular, we may assume that $N$ is a Borel set. By, for example,  Lemma~423G in \cite{Fremlin4}, we see that $f(N)$ is analytic. We apply Theorem~57 in \cite{Rogers} to obtain a compact set $M\subset f(N)$ with positive and finite $\H^n$-measure. We look at the closed set $f^{-1}(M)$ and by decomposing if necessary, we may further  assume that $N$ is a set of $\mathcal{H}^\varphi$\nobreakdash-measure zero contained in a compact set $K_0$ (not necessarily of measure zero) with \mbox{$\mathcal{H}^n(f(N))\leq\mathcal{H}^n(f(K_0))<\infty$}.

We start with verifying the existence of a compact set $K\subset\Omega$ with $\mathcal{H}^\varphi(K)=0$ and $\mathcal{H}^n(f(K))>0$.

Let us fix a sequence $(a_m)_m$
decreasing to zero such that $a_1<1$ and $\prod_{m=1}^\infty (1-a_m)\geq \frac{1}{2}$.

We may cover $N$ with countably many closed balls $\overline B_i$
such that $\diam {\overline B}_i<a_1$ and $\sum_i \varphi(\diam \overline B_i)<a_1$.

We set
$$L_i:=K_0\cap\bigcup_{k=1}^i \overline B_k
$$
and note that
$$
\mathcal{H}^n\Bigl(\bigcup_i f(L_i\cap N)\Bigr)=\mathcal{H}^n(f(N))>0.
$$
Since the sequence of the sets $L_i$ is increasing, we find $i_1$ such that
\begin{equation*}
\mathcal{H}^n\Bigl(f\Bigl(\bigcup_{i=1}^{i_1}L_i\cap N\Bigr)\Bigr)\geq (1-a_1)\mathcal{H}^n(f(N)).
\end{equation*}
We set
$$
K_1:=\bigcup_{i=1}^{i_1}L_i.
$$
Then $K_1$ is clearly compact.

Let us argue now that we can find a sequence $K_1\supset K_2\supset \cdots$ of nested compact sets with the property that $\mathcal{H}^\varphi(K_m)<a_m$ for $m\geq 1$ and
\begin{equation*}
\mathcal{H}^n(f(K_M))\geq \mathcal{H}^n(f(N\cap K_M))\geq \prod_{m=1}^M (1-a_m) \mathcal{H}^n(f(N))\geq\frac{1}{2}\mathcal{H}^n(f(N)).
\end{equation*}
Having constructed $K_1,\ldots, K_M$, we obtain $K_{M+1}$ by applying above procedure for $N\cap K_M$ and $a_{M+1}$ instead of $N$ and $a_1$, respectively, and intersecting the obtained set by $K_M$.

Let us now define the compact set $K:=\bigcap_{k=1}^\infty K_k$. For all $k\in \field{N}$, we have
\begin{equation*}
\mathcal{H}^\varphi_{a_k}(K)\leq \mathcal{H}^\varphi_{a_k}(K_k)<a_k.
\end{equation*}
Thus $\mathcal{H}^\varphi(K)=0$.

To show that $\mathcal{H}^n(f(K))>0$, we apply \tref{intersection of compacts} and conclude that
\begin{equation*}
\mathcal{H}^n(f(K))=\mathcal{H}^n(\bigcap_{m=1}^\infty f(K_m))=\lim_{m\to \infty} \mathcal{H}^n(f(K_m))\geq \frac{1}{2}\mathcal{H}^n(f(N))>0.
\end{equation*}

Next, we want to show that we can choose $K$ to be totally disconnected. The
compact set constructed in the preceding step will be relabelled as $K'$.
We may assume that $K'$ is a subset of a dyadic cube $I_n$. We
claim first that we can find $1\leq m \leq n$ and a closed dyadic
$m$-dimensional cube $I$ such that
\begin{itemize}
\item[(a)\label{Measure image of cube positive}] $\mathcal{H}^n(f(I\cap K'))>0$, and
\item[(b)] \label{Subcube} $\mathcal{H}^n(f(\partial J\cap K'))=0$ for every $m$\nobreakdash-dimensional
dyadic subcube of $J$ of $I$.
\end{itemize}
Indeed, letting $m=n$,
if $I=I_m$ does not satisfy (b), then there exists a $m$\nobreakdash-dimensional
dyadic subcube
$J_m$ and a ($m{-}1$)\nobreakdash-dimensional
face $I_{m-1}$ of $J_m$ such that
$\mathcal{H}^n(f(\partial I_{m-1}\cap K'))>0$.
We continue this process with $I=I_{m-1}$
until (b) holds for $I=I_k$ for some $k$.
It is clear that the process has to stop for if $m=1$, then the boundaries of the subcubes under consideration are points and therefore mapped to sets of $\mathcal{H}^n$\nobreakdash-measure zero.

We may and will in the following assume that $K'=I\cap K'$. Let $(J_i)$ be an enumeration of all dyadic subcubes of $I$. We fix $0<\varepsilon<\mathcal{H}^n(f(K'))$. The set $K'\setminus \bigcup_i \partial J_i$ is
totally disconnected but not necessarily compact as
the sets $\partial J_i$ are not open. However, choosing $J=J_i$, we can write $\partial J_i$ as intersections $\bigcap_m U_{m}=\bigcap_m \overline{U_m}$ where $U_m$ is the
$1/m$\nobreakdash-neighborhood of $J_i$. Now
\begin{equation*}
\partial J\cap K'=\bigcap_m \overline{U_{m}}\cap K'=\bigcap_m(\overline{U_{m}}\cap K')
\end{equation*}
and thus, by \tref{intersection of compacts}, we see that
\begin{equation*}
0=\mathcal{H}^n(f(\partial J\cap K'))=\mathcal{H}^n\Bigl(f\Bigl(\bigcap_m(\overline{U_{m}}\cap K')\Bigr)\Bigr)=
\mathcal{H}^n\Bigl(\bigcap_m f(\overline{U_{m}}\cap K')\Bigr).
\end{equation*}
We can thus choose an open set $O_i$ such that $\partial J\subset O_i$ and $\mathcal{H}^n(f(K'\cap O_i))<\varepsilon/2^i$.

By Cantor-Bendixson's theorem, see for example Theorem~XIV.5.3 in \cite{Kuratowski}, we may write any closed set as union of a perfect and a countable set.
We extract such a perfect set from $K'\setminus \bigcup_i O_i$ and call it $K$.

It is clear that $K$ is compact and that $\mathcal{H}^\varphi(K)=0$. By the choice of the sets $O_i$, we further have $\mathcal{H}^n(f(K))>0$, and by above choice, $K$ is perfect.

To show that $K$ is totally disconnected, let us assume that $A\subset K$ and $x$ and $y$ are two points in $A$. Then there is a dyadic cube $J$ containing $x$ but not $y$. Then the set $A\setminus \partial J=A$ is the union of two disjoint open sets one containing $x$ and the other $y$. Thus $A$ is not connected. We conclude that $K$ is totally disconnected.
\end{proof}

\bibliographystyle{amsalpha}
\bibliography{Hoelder}
\end{document}